\documentclass[11pt]{amsart}
\usepackage{a4wide, amsfonts, amsmath, amssymb, amsthm, epsfig,cite, graphicx, hyperref, color,esint, fancyhdr, enumerate,latexsym, amsrefs, mathtools}

\newtheorem{thm}{Theorem}[section]
\newtheorem{prop}[thm]{Proposition}

\newtheorem{cor}[thm]{Corollary}
\theoremstyle{definition}

\newtheorem{rem}[thm]{Remark}
\newtheorem*{claim}{Claim}

\newcommand{\de}{\delta}

\newcommand{\pa}{\partial}

\newcommand{\Om}{\Omega}

\newcommand{\be}{\beta}

\newcommand{\ti}{\tilde}
\newcommand{\De}{\Delta}

\renewcommand{\Re}{\operatorname{Re}}

\numberwithin{equation}{section}

\newcommand\restr[2]{{
  \left.\kern-\nulldelimiterspace 
  #1 
  \littletaller 
  \right|_{#2} 
  }}

\newcommand{\littletaller}{\mathchoice{\vphantom{\big|}}{}{}{}}

\hypersetup{
    colorlinks,
    citecolor=blue,
    filecolor=black,
    linkcolor=red,
    urlcolor=black
}

\title{Limits of an increasing sequence of Riemann surfaces}
\author{Diganta Borah, Prachi Mahajan and Jiju Mammen}

\address{Diganta Borah: Indian Institute of Science Education and Research Pune, Pune  411008, India}
\email{dborah@iiserpune.ac.in}

\address{Prachi Mahajan: Department of Mathematics, Indian Institute of Technology Bombay, Powai, Mumbai 400076, India}
\email{prachi.mjn@iitb.ac.in}

\address{Jiju Mammen: Department of Mathematics, Indian Institute of Technology Palakkad, 678557, India}
\email{211914001@smail.iitpkd.ac.in}

\begin{document}
\maketitle
\begin{abstract}
Let $M$ be a Riemann surface which admits an exhaustion by open subsets $M_j$ each of which is biholomorphic to a fixed domain $\Omega \subset \mathbb{C}$. We describe $M$ in terms of $\Omega$ under various assumptions on the boundary components of $\Omega$. 
\end{abstract}

\maketitle

\section{Introduction}

Let $ M $ be a complex manifold, exhausted by
open Stein submanifolds $ \{M_j\}_{j=1}^{\infty}$ such that $ M_1 \subset M_2 \subset \ldots $. Is $ M $ Stein? 

The above question was posed by Behnke and Thullen \cite{BT} in 1933. An affirmative answer to this question was provided by Behnke and Stein \cite{BS-39} in case $ M $ is an open subset of $ \mathbb{C}^n$, for $ n \geq 1 $. Indeed, Behnke and Stein proved  that the union of an increasing sequence of domains of holomorphy is a domain of holomorphy. This result was a crucial ingredient in solving the classical Levi problem---whether a pseudoconvex domain in $ \mathbb{C}^n $ is a domain of holomorphy. On the other hand, Forn\ae ss showed that the answer to the above question is negative in general by providing explicit counter-examples, first in dimension $\geq 3$ \cite{For-76} and then in dimension $2$ (see \cite{For-77}). It is to be noted that in dimension one, the above question reduces to a topological one by a result of Behnke-Stein \cite{BS-49}---a Riemann surface is Stein if and only if it is non-compact---from which it follows immediately that the answer to the above question is positive. The counterexamples due to Forn\ae ss were the rationale behind studying the following more general problem:

Let $ M $ be a complex $ n$-dimensional manifold admitting an exhaustion by open subsets, each of which is biholomorphic to a fixed domain $ \Omega $ in $ \mathbb{C}^n $. Determine $ M $ in terms of $ \Omega $.

This problem has been investigated by many authors under diverse assumptions on the domain $\Om$. We refer to the papers \cite{BBMV}, \cite{Behrens}, \cite{DS}, \cite{FP}--\cite{Fridman2}, and  \cite{Liu}, and the references therein. To the best of our knowledge, the problem when $\Omega$ has punctures (i.e., isolated boundary points) has not been studied, possibly because in dimension $ n \geq 2$, every isolated singularity of a holomorphic function is removable. Thus, it is of independent interest to investigate this problem in dimension $n=1$ to accommodate such domains which is the motivation behind this note. To summarise, the purpose of this article is to describe Riemann surfaces $ M$ which are unions of increasing sequences of open subsets $ M_j $, each $M_j$ being biholomorphically equivalent to a fixed domain $ \Omega \subset \mathbb{C} $. Henceforth, this problem will be referred to as the \textit{union problem in dimension one}. 

\begin{thm} \label{T1}
Assume that in the union problem, $ M $ is a hyperbolic Riemann surface. 

\begin{enumerate}
 \item [(i)] If $ \Omega $ is a bounded domain in $ \mathbb{C} $ with $C^2$-smooth boundary, then  $ M $ is biholomorphic either to $ \Omega $ or to the unit disc $ \Delta $.

 \item [(ii)] If $ \Omega = D \setminus A $, where $ D $ is a bounded domain in $ \mathbb{C} $ with $C^2$-smooth boundary and $ A \subset D $ is finite, then  $ M $ is biholomorphic either to $ \Omega $ or to the unit disc $ \Delta $.
 
\end{enumerate}
 
\end{thm}

Some remarks are in order. First, if the Riemann surface $M$ in the union problem is non-hyperbolic, it follows from the uniformization theorem that $M$ is either $\mathbb{C}$ or the punctured plane $\mathbb{C}^{*}=\mathbb{C} \setminus \{0\}$. Furthermore, it turns out that $ M $ must be biholomorphic to $ \mathbb{C} $ if the domain $ \Omega $ is smoothly bounded (see Subsection~\ref{S2}). These two disparate cases---$M$ non-hyperbolic and hyperbolic---are inspired by Forn\ae ss and Sibony's work \cite{FS1981}. They demonstrated that the notion of Kobayashi hyperbolicity can be employed effectively to explore certain aspects of the union problem. 

Second, the domains $ \Omega $ dealt with in Theorem \ref{T1} are all finitely connected, i.e, their complement in the Riemann sphere $ \mathbb{C}_{\infty} $ have finitely many connected components. What happens to $M$ when $\Omega$ is infinitely connected is not clear as there is no canonical choice for $\Omega$ in this case. We attempt to understand this situation in Section \ref{S6}, by considering two explicit examples of infinitely connected domains $ \Omega $.  

Finally, the next theorem follows from Theorem~\ref{T1} and the remarks made in Subsection~\ref{S2} but is worth explicit mention here. 

\begin{thm} \label{T5} In the union problem in dimension one, 
\begin{enumerate}
    \item[(i)] if $ \Omega = \Delta $, then $ M $ is biholomorphic either to $ \Delta $ or to $ \mathbb{C} $.
    \item[(ii)] if $ \Omega = \{ z \in \mathbb{C}: r < |z |< 1 \}$, $ 0 < r<1$, then $ M $ is biholomorphic to one of $ \Omega $, $ \Delta $, or $\mathbb{C}$.    
    \item[(iii)] if $ \Omega = \mathbb{C}$, then $ M $ is biholomorphic to $  \mathbb{C} $.
    \item [(iv)] if $\Omega=\mathbb{C} \setminus \{a_1, \ldots, a_n\}$, $n \geq 1$, then $M$ is biholomorphic to $\Omega$.
    \end{enumerate}
\end{thm}

\section{Preliminaries}

\subsection{Some notations:}

\begin{itemize}
 
 \item $ \Delta(a;r) $ is the open disc centered at point $ a \in \mathbb{C} $ and radius $ r > 0 $. The unit disc $ \Delta (0;1) $ is written as $ \Delta $
 for simplicity.

 \item $ \Delta^* $ is the punctured unit disc $ \Delta \setminus \{0\} $.

 \item $\mathbb{C}^{*}$ is the punctured plane $\mathbb{C} \setminus \{0\}$.
 
 \item For non-empty sets $ \Omega_1, \Omega_2 \subset \mathbb{C} $, 
 \[ \mbox{dist}( \Omega_1, \Omega_2) := \inf \big\{ |z-w| : z \in \Omega_1,  w \in \Omega_2 \big\},
 \]
where $ |\cdot| $ denotes the standard Euclidean norm on $ \mathbb{C} $.

\item $\mathbb{C}_{\infty}$ is the Riemann sphere.

\item If $\Omega \subset \mathbb{C}$, then $\pa_{\infty}\Omega$ is the boundary of $\Omega$ in $\mathbb{C}_{\infty}$.
 
\end{itemize}

\subsection{Some definitions:}\label{S2-0} Recall that a complex manifold $ X $ is said to be Kobayashi hyperbolic (or simply hyperbolic) if the
Kobayashi pseudo-distance $ d_X $ is actually a distance: namely, for all $ z,w \in X $, $ d_X (z,w) = 0 $ if and only if $ z=w $. It is well known (see \cite[Theorem~2]{Roy}) that $ X $ is hyperbolic if for each point $ z \in X $, there is a neighbourhood $ U $ of $ z $ and a positive constant $ c $ such that the infinitesimal Kobayashi metric $ k_X $ satisfies
\begin{equation} \label{E9}
    k_X(w, \xi) \geq c \; \| \xi\|, 
\end{equation} 
for all $ w \in U $ and for each holomorphic tangent vector $ \xi $ at the point $ w $.

It is worthwhile mentioning that the notion of Kobayashi hyperbolicity coincides with the notion of hyperbolicity in the context of Riemann surfaces - any Riemann surface
uniformized by the unit disc is Kobayashi hyperbolic.

We also recall that a complex manifold $X$ is taut if the family $ \mathcal{O}(\Delta, X ) $ of holomorphic maps from $ \Delta $ to $ X $ is normal, i.e., each sequence in $ \mathcal{O}(\Delta, X ) $ admits a subsequence that is either uniformly convergent on compact subsets of $ \Delta $ or compactly divergent. It is well-known that a Riemann surface $X$ is taut if and only if it is hyperbolic. Indeed, if $X$ is taut, then by Proposition~5 of \cite[Page~135]{Roy}, $X$ is hyperbolic. Conversely, if $X$ is a hyperbolic Riemann surface then it is complete as its universal cover $\Delta$ is complete (see Theorem~4.7 of \cite[Chapter IV]{Kob}), and hence the corollary in \cite[Page~136]{Roy}, implies that $X$ must be taut.

A sequence of domains $ G_j \subset \mathbb{C}^n $ is said to converge to a domain $ G \subset \mathbb{C}^n $ in the local Hausdorff sense if and only if the following two conditions hold:
\begin{itemize}
    \item Given any compact set $ K \subset G $, there is a positive integer $m$ such that for every $j \geq m$, $ K \subset G_j $, and
    \item If a compact set $ K $ is contained in the interior of $ \cap_{j \geq m}G_j $ for some positive integer $m $, then $ K \subset G $.
\end{itemize}

Lastly, we also recall the concept of a cluster set of a function $f: G \to \mathbb{C}$, where $G \subset \mathbb{C}$ is an open set, at a point $a \in \pa_{\infty}G$. The cluster set of $f$ at $a$ is the set
\begin{multline*}
\text{Clu}(f;a)=\Big\{\zeta \in \mathbb{C}_{\infty} : \text{there exists a sequence $\{z_n\} \subset G$}\\
\text{such that $z_n \to a$ and $f(z_n) \to \zeta$ in $\mathbb{C}_{\infty}$}\Big\}.
\end{multline*}
For more on this concept we refer to \cite[Section 14.1]{Con2}.

 \subsection{Some observations about the union problem:} \label{S2}

Firstly, it is straightforward to see that if $ \Omega $ is simply connected, then $ M $ is also simply connected.

Secondly, if $ \Omega $ is non-hyperbolic, then so is $ M $. Indeed, if $ \Omega $ is non-hyperbolic, then each $ M_j $, being a biholomorph of $ \Omega $, is also non-hyperbolic. Then the distance-decreasing property for the Kobayashi pseudo-distance under the inclusion map $ M_j \xhookrightarrow{} M $ forces that $ M $ cannot be hyperbolic.

Thirdly, the union $ M $ can never be compact. We shall establish this by contradiction. Assume that $ M $ is compact. Then, $\{M_j\}_{j=1}^{\infty}$ being an open cover of $M$, must admit a finite subcover. This, together with the fact that $ \{M_j\} $ is increasing, implies that there exists an integer $N\in \mathbb{N}$ such that $M_N=M$. Therefore $M$ is biholomorphic to $\Omega$ which is a contradiction as $\Omega$ is non-compact.  

Finally, since $ M $ is necessarily non-compact, it follows from the uniformization theorem for Riemann surfaces that $ \mathbb{C}_{\infty} $ cannot be the universal covering space of $ M $. As a consequence, it is immediate that the universal covering space of $ M $ is either $ \Delta $ or $ \mathbb{C} $. In particular, if $ M $ is non-hyperbolic, it follows that $ M $ must be covered by $ \mathbb{C} $. It is well-known that the only Riemann surfaces having $ \mathbb{C} $ as the universal covering surface are $ \mathbb{C} $ itself, the punctured complex plane $ \mathbb{C}^{*} $, and tori. The third case is ruled out by the fact that $M$ is non-compact and complex tori are compact. Hence, if $ M $ is a non-hyperbolic Riemann surface, then $ M $ is biholomorphic either to $ \mathbb{C} $ or to $ \mathbb{C}^{*} $.
Moreover, under the additional hypothesis that $ \Omega $ is a bounded domain with $C^2$-smooth boundary, then only the case $ M $ biholomorphic to $ \mathbb{C} $ is possible. To see this, note that \cite[Main Theorem]{FS1981} or \cite[Theorem 1.5]{Behrens} imposes that $ M $ is biholomorphic to a locally trivial holomorphic fibre bundle with fibre $ \mathbb{C} $ over a retract $ Z $ of $ \Omega $ or a retract $ Z' $ of $ \Delta $. Recall that  holomorphic retracts of a planar domain are either points or the whole domain itself. If $ Z = \Omega $ or $ Z' = \Delta $, then $ M $ will be biholomorphic to a $2$-dimensional manifold, which is impossible. Thus, $ Z $ and $ Z' $ must be singletons. It follows that $ M $ must be biholomorphic to $ \mathbb{C} $.

\section{When $ M $ is hyperbolic} \label{S4}

\subsection {An useful result: }
Let us first note the following version of a result due to Joseph and Kwack \cite[Theorem 1]{JK-1994}, which will be used in the sequel.

\begin{prop} \label{T4}\textup{(special case of the implication (1) $\Rightarrow$ (3) of Theorem 1, \cite{JK-1994})} 
Let $ X_j $ be an increasing sequence of open subsets of a Kobayashi hyperbolic complex manifold $ X $ that exhausts $ X $, i.e., $ X = \cup_{j=1}^{\infty} X_j $. If $ f_j : \Delta^* \rightarrow X_j $ is a sequence of holomorphic mappings and $ \{ z_j \} $ is a sequence in $ \Delta^* $ such that $ z_j \rightarrow 0 $ and $ f_j(z_j) \rightarrow z_0 \in X $, then the following two conditions hold:

\begin{itemize}
    \item Eventually each $ f_j $ extends to a holomorphic mapping $ {F_j} : \Delta \rightarrow X $, and
    \item there is an $ 0 < r < 1 $ and a subsequence $ \{ F_{j_k} \} $ of $ \{ F_j \} $ such that $ \{ F_{j_k} \} $ converges uniformly on compacts of $ \Delta(0;r) $ to a holomorphic mapping $ F : \Delta(0;r) \rightarrow X $. 
\end{itemize}
    
\end{prop}

The authors in \cite{JK-1994} introduced a notion of a \textit{hyperbolic point}. With $ X_j, X $ as in Proposition \ref{T4}, the distance-decreasing property for the Kobayashi metric under the inclusion map $ X_j \xhookrightarrow{} X $
together with hyperbolicity of $ X $ (cf. \eqref{E9}) ensures that $ z_0 \in X $ is a \textit{hyperbolic point} of $ X_j $ for each $ j$ - as required in \cite[Theorem 1]{JK-1994}. From this point, an argument similar to the one employed in \cite[Theorem 1]{JK-1994} can be used to complete the proof of Proposition \ref{T4}. We will provide an alternative proof of a stronger result in dimension one, i.e., when $X$ is a Riemann surface, later (see Proposition~\ref{extn}).

\medskip

We now begin with the proof of Theorem \ref{T1}.
 
\subsection{Proof of Theorem \ref{T1}:}

Let $ \psi_j : M_j \rightarrow \Omega $ be biholomorphisms and $ z_0 \in M $ be fixed. Since $ M = \cup \; M_j $, we may assume that $ z_0 \in M_j $ for all $ j$.  The proof bifurcates into two cases:
\begin{enumerate}
 \item[(a)] $ \{\psi_j(z_0) \} $ is a relatively compact subset of $ \Omega $, or
 \item[(b)] $ \{ \psi_j(z_0) \} $ has at least one limit point $ p_0 \in \partial \Omega $.
\end{enumerate}
In case (a), since the domain $\Omega$ is bounded, it follows from \cite[Lemma~1.1]{Fridman} that $ M $ is biholomorphic to $ \Omega $. 

It remains to investigate the case (b). Here, some
subsequence of the `orbit' $ \{\psi_j(z_0)\} $ (which we continue to denote by the same symbols) accumulates at $ p_0 \in \partial \Omega $.

\subsubsection{\textup{\textbf{Proof of Theorem \ref{T1} (i):}}} Since $ \partial \Omega $ is $ C^2$-smooth near the point $ p_0 $, the scaling technique (see \cite{GKK} for more details) applies to yield a sequence $ \Omega_j$ of planar domains, which converge in local Hausdorff sense to $ \Delta $. Indeed, let us scale $ \Omega $ with respect to $ \{ \psi_j(z_0) \} $, i.e. apply the affine maps
\begin{equation*}
T_j(z) = \frac{ z - \psi_j(z_0) }{ -\rho \big(\psi_j(z_0)\big) },
\end{equation*}
where $ \rho $ is a a $ C^2 $-smooth local defining function for $ \Omega $ near $ p_0 $. Writing $ p_j = \psi_j(z_0) $ for brevity, note that a defining function for the domain $ \Omega_j := T_j( \Omega) $ near $ T_j(p_0) $ is given by
\begin{align*}
    \frac{1}{- \rho(p_j)} \rho \circ T_j^{-1} (z) & =  \frac{1}{- \rho(p_j)} \rho \big( p_j - \rho(p_j)z \big) \\
    & =\frac{1}{- \rho(p_j)} \left( \rho(p_j)-2\rho(p_j) \Re \big( \partial \rho(p_j) z\big)+ \big( \rho(p_j) \big)^2 O(1)  \right)\\
   & = -1 + 2 \Re \big( \partial \rho(p_j) z\big) -\rho(p_j) O(1).
\end{align*}
It follows that the domains $ T_j(\Omega) $ converge to the half-plane
\begin{equation*} 
\mathbb{H} = \big\{ z \in \mathbb{C} : \Re \big( \partial \rho(p_0) z\big) < 1/2 \big\},
\end{equation*}
in the local Hausdorff sense. Moreover, the limiting domain $ \mathbb{H} $ is conformally equivalent to $ \Delta $ via a M\"{o}bius transformation. 

Then, it follows that (see, for instance, \cite[Theorem~1.2]{Behrens} or \cite[Theorem 8.1]{BBMV}) $ M $ is biholomorphic to $ \Delta $. This completes the proof of Theorem \ref{T1} (i). 

\subsubsection{\textup{\textbf{Proof of Theorem \ref{T1} (ii):}}} Here, we first establish that $ \{ \psi_j(z_0) \} $ cannot accumulate at any point of $ A $, i.e. $ p_0 \notin A $.

Assume, on the contrary, that $ \psi_j(z_0) \rightarrow p_0 \in A $. To recall, here, $ \psi_j : M_j \rightarrow \Omega $ are biholomorphisms. Then, the inverses $ \phi_j := ( \psi_j)^{-1} : \Omega \rightarrow M_j $ satisfy $ \phi_j \circ \psi_j (z_0 ) = z_0 $ for all $ j $. In this setting, Proposition \ref{T4} implies that each $ \phi_j $ extends across the `puncture' $ p_0 $ to a holomorphic mapping $ {\Phi_j} : \Omega \cup \{p_0\}  \rightarrow M $. 

Let us write $ {\Phi_j} (p_0) = q_j $, and note that $ q_j \notin M_j $. Indeed, if $ q_j \in M_j $, then $ {\Phi_j} $ assumes the value $ q_j $ at $ p_0 $ as well as at $ \psi_j(q_j) \neq p_0 $. But then by the open mapping theorem, $ \phi_j $ must assume values sufficiently close to $ q_j $ at points arbitrarily close to $ p_0 $ as well as $ \psi_j(q_j) $. This is a contradiction since $ \phi_j $ is injective on $ \Omega $.

Furthermore, Proposition \ref{T4} also guarantees the existence of a neighbourhood $ U $ of $ p_0 $, $ U $ is relatively compact in $ D $, $ U \cap A = \{p_0\} $, with the property that some subsequence, say $ \{ \Phi_{j_k} \} $, of $ \{ \Phi_j \} $ converges uniformly on compacts of $ U $ to a holomorphic mapping $ {\Phi} : U \rightarrow M $. In particular, it follows that
\begin{equation*}
{\Phi}(p_0) = \lim_{ k \rightarrow \infty} { \Phi}_{j_k}  \big(  \psi_{j_k} (z_0 ) \big) = \lim_{ k \rightarrow \infty} { \phi}_{j_k}  \big(  \psi_{j_k} (z_0 ) \big) = z_0. 
\end{equation*}

\begin{claim}
$ \{ q_{j_k} \}_{k=1}^{\infty} $ cannot be compactly contained in $ M $.
\end{claim}

\begin{proof}[Proof of the claim] We establish the above claim by contradiction. Suppose that $ \{ q_{j_k} \}_{k=1}^{\infty} $ is compactly contained in $ M $. Then, since $ M = \cup \; M_j $, there must be an $ l_0 \in \mathbb{N} $ such that $ q_{j_k} \subset M_{l_0} $ for all $ k \in \mathbb{N} $. Moreover, since $ q^{j_k } \notin M_{j_k}  $ and $ M_j \subset M_{ j+1 } $, it is immediate that 
\[ l_0 > j_k \geq k \]
for all $ k \in \mathbb{N} $, which is evidently a contradiction. Hence the claim follows. \phantom\qedhere \hfill $ \blacktriangleleft $
\end{proof}
On the other hand, notice that 
\[
q_{j_k} = {\Phi_{j_k}} (p_0) \rightarrow {\Phi} (p_0) = z_0 \in M. 
\] 
Clearly, the above observation contradicts the claim. Hence, $ \{ \psi_j(z_0) \} $ cannot cluster at any point of $ A $.

\medskip

Following Proposition 2.2 and Corollary 2.3 of  \cite[Chapter 6]{Kob}, under the hypothesis of Proposition~\ref{T4}, it is possible to draw a stronger conclusion in dimension one---the map $f_j:\Delta^{*} \to X_j$ extends to a map $F_j : \Delta \to X_j$ for all $j$ large---based on which we now provide an alternative argument that the orbit cannot accumulate at a puncture. First, we prove this stronger version of the extension result.

\begin{prop}\label{extn}
Let $ X_j $ be an increasing sequence of open subsets of a hyperbolic Riemann surface $ X $ that exhausts $ X $, i.e., $ X = \cup_{j=1}^{\infty} X_j $. If $ f_j : \Delta^* \rightarrow X_j $ is a sequence of holomorphic mappings and $ \{ z_j \} $ is a sequence in $ \Delta^* $ such that $ z_j \rightarrow 0 $ and $ f_j(z_j) \rightarrow z_0 \in X $, then the following two conditions hold:

\begin{itemize}
    \item [(i)] Eventually each $ f_j $ extends to a holomorphic mapping $ {F_j} : \Delta \rightarrow X_j $, and
    \item [(ii)] there is a subsequence $ \{ F_{j_k} \} $ of $ \{ F_j \} $ such that $ \{ F_{j_k} \} $ converges uniformly on compacts of $ \Delta$ to a holomorphic mapping $ F : \Delta \rightarrow X $. 
\end{itemize}
\end{prop}
\begin{proof}[Proof of Proposition]
Consider the circles
\[
\sigma_j = \big\{ z \in \Delta^{*} : |z | = |z_j | \big\}.
\]
Following the arguments in the proof of Proposition 2.2 of \cite[Chapter~6]{Kob}, we show that for all $j$ large, $ f_j $ maps the circle $ \sigma_j $ into a closed curve which is homotopic to zero in $M_j$. Indeed, first recall that $ f_j (z_j) \rightarrow z_0 $ for all $ j $. So Lemma~2 of \cite{JK-1994} applies, and therefore,
\begin{equation} \label{E1}
f_j ( \sigma_j) \rightarrow z_0,
\end{equation}
as $ j \rightarrow \infty $. Now, if $ V $ is any relatively compact simply connected neighboourhood of $ z_0 $ in $ M $, then for all $j$ large, $ V \subset M_j $. Also, from (\ref{E1}), for all $j$ large, we have $ f_j ( \sigma_j)  \subset V $. Since $ V $ is simply connected, it follows that for all $j$ large, the closed curve $f_j ( \sigma_j) $ is nullhomotopic in $ M_j $. Since each $\sigma_j$ is a generator of $\pi_1(\De^{*})$, we conclude that, for all $j$ large, $ (f_j)_* \big( \pi_1 ( \Delta^{*})\big)$ is trivial at the level of fundamental groups. 

Next, observe that each $X_j $, being a subset of a hyperbolic Riemann surface $X$, is itself hyperbolic, and hence $ \Delta $ is the universal cover for $ X_j $. Let $ \pi_j: \Delta \rightarrow X_j $ be a covering projection from $ \Delta $ onto $ X_j $. It should be noted that the above observation is exactly the assertion
\[
(f_j)_* \big( \pi_1 ( \Delta^{*}) \big) \subseteq (\pi_j)_* \big( \pi_1 ( \Delta) \big),
\]
for all $ j $ large. As a consequence, $f_j $ lifts to $  \ti f_j : \Delta^{*}  \rightarrow \Delta $, i.e.
\begin{equation} \label{E2}
f_j = \pi_j \circ \ti f_j,
\end{equation}
holds for all $ j $ large. Moreover, since $ \ti f_j $ is bounded, $ \ti f_j $ extends across the `puncture' $ 0 $ to a holomorphic mapping $\ti F_j: \Delta \to \overline{\Delta} $. We note that $ \ti F_j (\Delta) \subset \Delta $. Indeed, otherwise, by the maximum principle, $\ti F_j$ is identically equal to a constant $\zeta \in \pa \De$, which in turn implies that $\ti f$ is identically equal to $\zeta$, which contradicts that $\ti f_j$ maps $\De^{*}$ into $\De$. Thus $ \ti F_j : \Delta \rightarrow {\Delta} $ is confirmed. It follows that $ F_j:=\pi_j \circ \ti F_j $ is the required extension of $f_j$ and the proof of (i) is complete.

For (ii), recall from Subsection~\ref{S2-0} that $X$ is taut. Since $F_j: \Delta \rightarrow X_j \subset X $, it follows from the tautness of $ X $ that there is a subsequence $\{F_{j_k}\}$ of $\{F_j\}$ that is either compactly convergent to a map $F: \Delta \to X$ or is compactly divergent. However the latter condition is ruled out by the fact that $F_j(z_j) \rightarrow z_0 \in X$, and hence (ii) is proved. \phantom\qedhere \hfill $ \blacktriangleleft $
\end{proof}

We will now exhibit an alternative argument that the orbit cannot accumulate at a puncture. Suppose that $ \psi_j(z_0) \rightarrow p_0 \in A $. Fix a tiny disc $ U $ around the point $ p_0 $ in $ D $ chosen so that $ U \cap A = \{p_0\} $. Recall that $ \phi_j \big( \psi_j (z_0 ) \big) = z_0 $ for all $ j $. So applying Proposition~\ref{extn} to $\phi_j\vert_{U \setminus \{p_0\}} $, it follows that $ \phi_j $ can be extended to a holomorphic mapping $ {\Phi_j} : \Omega \cup \{p_0\}  \rightarrow M_j $. 

To conclude, write $ {\Phi_j} (p_0) = q_j $, and note that $ q_j \in M_j $ by construction. But then, $ {\Phi_j} $ assumes the value $ q_j $ at $ p_0 $ as well as at $ \psi_j(q_j) \neq p_0 $. This, in turn, implies that $ \phi_j $ must assume values sufficiently close to $ q_j $ at points arbitrarily close to $ p_0 $ as well as $ \psi_j(q_j) $. This is a contradiction since $ \phi_j $ is injective on $ \Omega $. Hence, the above provides an alternate argument for the claim that $ \{ \psi_j(z_0) \} $ cannot acummulate at any point of $ A$. 

\medskip

To continue with the proof of the theorem: It is evident from the above discussion that we only need to consider the case when $ \{ \psi_j(z_0) \} $ accumulates at $ p_0 \in \partial D $. Since $ \partial D $ is $ C^2$-smooth near the point $ p_0 $, the proof proceeds exactly as in Theorem \ref{T1} (i) using the scaling method to yield that $ M $ is biholomorphic to $ \Delta $. 
This completes the proof of the theorem. \qed


\section{When $\Omega$ has infinite connectivity} \label{S6}

Throughout this section, $ M $ is a hyperbolic one-dimensional complex manifold.
As mentioned in the introduction, we do not know if, in the union problem, it is possible to describe $M$ explicitly when $\Omega$ is infinitely connected. We examine the problem for two specific infinitely connected domains $\Omega$ to see some possible scenarios. 

\subsection{When $ \Omega = \Delta \setminus A $, where $ A \subset \Delta $ is discrete and infinite:} \label{SS1} Evidently, $ A $ is countable. Let $ \{a_n: n = 1, 2, \ldots \} $ be an enumeration of $ A $. Since $ A $ is infinite and $ \Delta $ is bounded, $ A $ will have accumulation points on
$ \partial \Delta $ and such points will be non-smooth boundary points on
$ \partial \Omega $. 

Keeping the notation from Section \ref{S4}, recall the biholomorphisms $ \psi_j : M_j \rightarrow \Omega$. As before, for $ z_0 \in M $
fixed, we need to examine two distinguished scenarios:
\begin{enumerate}
 \item[(a)] $ \{\psi_j(z_0) \} $ is a relatively compact subset of $ \Omega $, or
 \item[(b)] $ \{ \psi_j(z_0) \} $ has at least one limit point $ p_0 \in \partial \Omega = \partial \Delta \cup A$.
\end{enumerate}
Much like before, in case (a), $ M $ turns out to be biholomorphic to $ \Omega $. Thus, we now focus on case (b).

In case (b), exactly as in the proof of Theorem \ref{T1}(ii), it can be shown that  $ \{ \psi_j(z_0) \} $ cannot cluster at any point of $ A $. As a consequence, $ \{ \psi_j(z_0) \} $ can only accumulate on points of $ \partial \Omega $ which lie on the unit circle. Now, if $ \{ \psi_j(z_0) \} $ accumulates at a smooth point of $ \partial \Omega $, then the scaling method - same as that employed in the proof of Theorem \ref{T1} - yields that $ M $ is biholomorphic to $ \Delta $. 

It remains to investigate the case when $ \{ \psi_j(z_0) \} $ accumulates at a non-smooth boundary point of $ \Omega $, say, $ p_0 \in  \partial \Delta $. This is precisely the case when a subsequence of punctures also converges to $p_0$. We do not have a definite answer in this case. But we consider a special case to see what $M$ may look like. Without loss of generality, we assume that $a_n \to p_0$ and $p_0=1$. Assume further that and $ 0 \leq \psi_j(z_0) < 1 $ for all $j$.

Let us scale $ \Delta $ with respect to $ \{ \psi_j(z_0) \} $, i.e. apply the affine maps
\begin{equation} \label{E6}
T_j(z) = \frac{ z - \psi_j(z_0) }{ -\rho \big(\psi_j(z_0)\big) },
\end{equation}
where $ \rho(z) = |z|^2 - 1$ is a defining function for $ \Delta $. Then, as in the proof of Theorem \ref{T1}(i), 
it can be seen that the domains $ T_j( \Delta) $ converge to the half-plane
\begin{equation} \label{E7}
\mathbb{H} = \{ z \in \mathbb{C} : \Re z < 1/2 \},
\end{equation}
in the Hausdorff sense. Write $ p_j = \psi_j(z_0) $ for brevity and set
\[
a_{nj} : = T_j(a_n) = \frac{a_n - p_j}{1- p_j^2}. 
\]
The goal is to understand the \textit{limit} of $ T_j(\Omega) = T_j(\Delta) \setminus \{ a_{nj}: n \in \mathbb{N} \} $, if it exists. To this end, first observe that
\begin{itemize}
    \item For each fixed $ n \in \mathbb{N} $, $ a_{nj} \rightarrow \infty $ as $ j \rightarrow \infty $.
    
    \item For each fixed $ j \in \mathbb{N} $, $ a_{nj} \rightarrow \frac{1}{1+p_j} \in \mathbb{C} \setminus \overline{\mathbb{H}} $ as $ n \rightarrow \infty $.
    
    \item The diagonal sequence
    \[
    a_{jj} = \frac{a_j - p_j}{1 - p_j^2}
    \]
    may accumulate inside or outside of $ \mathbb{H} $ depending on the asymptotic behaviour of $ |a_j - p_j| $. For instance, if $ |a_j - p_j | \lesssim \big( 1-p_j^2 \big)^2 $, then $ a_{jj} \rightarrow 0 $. On the other hand, if $ |a_j - p_j| \gtrsim \sqrt{1-p_j} $, then 
    $ a_{jj} \rightarrow \infty $. It is also possible that the diagonal sequence $a_{jj}$ is constant. Indeed, if 
    \[
    a_n=p_n+\frac{1}{4}(1-p_n^2),
    \]
    then $a_{jj}=1/4 \in \mathbb{H}$ for all $j$.
\end{itemize}
In particular, the set $ {A}^* := \{ a_{nj}: n, j \in \mathbb{N} \} $ may have limit points in $ \mathbb{H} $.

\smallskip

\noindent \textbf{Claim:} $ T_j( \Omega) \rightarrow \mathbb{H} \setminus \overline{A^{*}} $ in the Hausdorff sense.

\smallskip 

\noindent \textit{Proof of claim:} Let $ K $ be a compact subset of $ \mathbb{H} \setminus \overline{A^{*}} $.  Since $ T_j ( \Delta) \to \mathbb{H} $, it is immediate that $ K $ is compactly contained in $ T_j( \Delta ) $ for all $ j $ large. Without loss of generality, assume that $ K \subset T_j( \Delta) $ for all $ j $. We need to show that $ K \subset T_j( \Omega ) $ for all $ j $ large. If possible, assume that this is false. Then, there are subsequences $(j_k)_{k=1}^{\infty}$ and $(n_k)_{k=1}^{\infty}$ such that $a_{n_kj_k} \in K$. Since $K$ is compact, passing to a subsequence, $a_{n_kj_k} \to a \in K$. Evidently, $a \in \overline{A^{*}}$. This contradicts that $K \cap \overline{A^{*}}=\emptyset$.

Now, let $ K $ be uniformly compactly contained in $ T_j( \Omega ) $ for all $ j $ large. Without loss of generality, assume that this is the case for all $ j $. This means that there exists $ \delta > 0 $ such that for each $j$,
\begin{equation} \label{E3}
\mbox{dist} \big( K, \mathbb{C} \setminus T_j(\Omega) \big) \geq \delta. 
\end{equation}
The goal is to establish that $ K \subset \mathbb{H} \setminus \overline{A^*}$. Towards this goal, we note that \eqref{E3} has the following two consequences. First, $K$ is uniformly compactly contained in $T_j(\Delta)$ for all $j$ and hence  $ K \subset \mathbb{H} $ as $ T_j( \Delta) \to \mathbb{H} $. Second, for each $j$,
\[
\mbox{dist} \big( K, \{ a_{nj}: n \in \mathbb{N} \} \big) \geq \delta,
\]
which implies that $\text{dist}(K, A^{*})\geq \de$. Therefore, $\text{dist}(K, \overline{A^{*}})\geq \de$ and so $ K \cap \overline{A^{*}} = \emptyset $. Thus $K \subset \mathbb{H} \setminus \overline{A^{*}}$ as required. Hence, the claim is true. \hfill $ \blacktriangleleft $

Finally, using the arguments presented in the proof of Theorem 1.1 of \cite{BBMV}, it can be shown that $ M $ is biholomorphic to $ \mathbb{H} \setminus \overline{A^{*}} $ in this case. 
\subsection{When $ \Omega = \Delta \setminus A $, where $ A \subset \Delta $ is a countable union of pairwise disjoint closed discs:} Let 
\[
A= \cup A_n = \cup_{n=1}^{\infty} \overline{\Delta(a_n;r_n)}
\]
be a union of pairwise disjoint closed discs $A_n=\overline{\Delta(a_n;r_n)}$ in $\Delta$ where $a_n \to 1$ and $r_n \downarrow 0$. It is evident that $ \partial \Omega $ is non-smooth near the point $ 1 $.

Again, following the notation from Section \ref{S4}, $ \psi_j : M_j \rightarrow \Omega$ are biholomorphisms and $ z_0 \in M $ is fixed. Two cases arise:
\begin{enumerate}
 \item[(a)] $ \{\psi_j(z_0) \} $ is a relatively compact subset of $ \Omega $, or
 \item[(b)] $ \{ \psi_j(z_0) \} $ has at least one limit point $ p_0 \in \partial \Omega = \partial \Delta \cup \left(\cup_{n=1}^{\infty} \partial A_n \right) $.
\end{enumerate}
As before, in case (a), $ M $ turns out to be biholomorphic to $ \Omega $. Thus, we are left with case (b).

In case (b), if $ \{ \psi_j(z_0) \} $ accumulates at a smooth point of $ \partial \Omega $, then the scaling technique yields that $ M $ is biholomorphic to $ \Delta $. 

Hence, we only need to address the case when $ \{ \psi_j(z_0) \} $ accumulates at the non-smooth boundary point $ 1 \in \partial \Omega $. For simplicity, assume that $ 0 \leq \psi_j(z_0) < 1 $ for all $j$. 

Scaling $ \Delta $ along the sequence $ \{ \psi_j(z_0) \} $ gives rise to a sequence of the scaled domains $ T_j ( \Delta) $, with $ T_j $ as defined by \eqref{E6}, 
that converge to the half-plane $ \mathbb{H} = \{ z \in \mathbb{C} : \Re z < 1/2 \}$, exactly as discussed in Section \ref{S4}. Setting
\[
A_{nj} : = T_j(A_n)= \left\{\frac{z - p_j}{1- p_j^2} :z \in \overline{\Delta(a_n;r_n)} \right\}, 
\]
we examine the \textit{limit} of $ T_j(\Omega) = T_j(\Delta) \setminus \cup_{n=1}^{\infty} A_{nj}$, provided it exists. Let $ {A}^* := \cup_{n,j}A_{nj} $.

\begin{claim}
$ T_j( \Omega) \rightarrow \mathbb{H} \setminus \overline{A^{*}}$ in the local Hausdorff sense.
\end{claim}

\begin{proof}[Proof of the claim]
Let $ K $ be a compact subset of $ \mathbb{H} \setminus \overline{A^{*}} $. Since $ T_j ( \Delta) $ converge to $ \mathbb{H} $, and $ K \subset \mathbb{H} $, it is immediate that $ K $ is compactly contained in $ T_j( \Delta ) $ for all $ j $ large. Therefore, in order to show that $ K $ is compactly contained in $ T_j( \Omega ) $ for all $ j $ large, it suffices to show that 
\begin{equation} \label{E4}
    K \cap \big(\cup _{n=1}^{\infty}  A_{nj} \big) = \emptyset 
\end{equation}
for all $ j $ large. We establish the above by contradiction. If (\ref{E4}) is false, then, there exist subsequences $(j_k)_{k=1}^{\infty}$ and $(n_k)_{k=1}^{\infty}$ such that for each $k$, $ K \cap A_{n_kj_k} \neq\emptyset$. It follows that there are points $a_{n_kj_k} \in K \cap A_{n_kj_k}$ for each $ k $. Since $K$ is compact, passing to a subsequence, $a_{n_kj_k} \to a \in K$. Evidently, $a \in \overline{A^{*}}$. This contradicts that $K \cap \overline{A^{*}}=\emptyset$.

Now, let $ K $ be uniformly compactly contained in $ T_j( \Omega ) $ for all $ j $ large. Without loss of generality, assume that this is the case for all $ j $. This means that there exists $ \delta > 0 $ such that for each $j$,
\begin{equation} \label{E5}
\mbox{dist} \big( K, \mathbb{C} \setminus T_j(\Omega) \big) \geq \delta. 
\end{equation}
The goal is to establish that $ K \subset \mathbb{H} \setminus \overline{A^*}$. To achieve this, note that \eqref{E5} has the following two consequences. First, $K$ is uniformly compactly contained in $T_j(\Delta)$ for all $j$ and hence  $ K \subset \mathbb{H} $ as $ T_j( \Delta) \to \mathbb{H} $. Second, for each $j$,
\[
\mbox{dist} \big( K, \cup_{n=1}^{\infty}A_{nj} \big) \geq \delta,
\]
which implies that $ \text{dist}(K, A^*) \geq \delta $. Therefore, 
$\text{dist}(K, \overline{A^*}) \geq \delta $ and so $ K \cap \overline{A^*} = \emptyset $. Thus $K \subset \mathbb{H} \setminus \overline{A^*}$ as required. Hence, the claim follows. 
\phantom\qedhere\hfill $ \blacktriangleleft $
\end{proof}

To conclude, note that as in Subsection \ref{SS1}, it follows that in this case, $ M $ is biholomorphic to $ \mathbb{H} \setminus \overline{A^{*}} $.

To summarise, when $\Omega$ in the union problem is infinitely connected, there does not seem to be any canonical description of $M$. 


\section{Some explicit examples} 

In this section, we present the proof of Theorem \ref{T5}. We begin with the following observation:

\begin{cor} \label{T3} 
Assume that in the union problem, $ \Omega \subset \mathbb{C} $ is such that $ \mathbb{C}_{\infty} \setminus \Omega $ has finitely many connected components and none of these components are singletons. Then $ M $ is biholomorphic to one of $ \Omega $, $ \Delta $ or $ \mathbb{C} $.
\end{cor}

Note that $ \Omega $ can \textit{a priori} be unbounded in Corollary \ref{T3}. It can be shown that the domains considered in Corollary \ref{T3} are biholomorphic to bounded domains with $ C^{\infty}$-smooth boundary (see, for instance, \cite{Ahlfors}). Hence, Corollary \ref{T3} is an immediate consequence of Theorem \ref{T1}(i) and remarks made in Subsection~\ref{S2}.

\begin{proof}[Proof of Theorem~\ref{T5}:] (i) If $M$ is hyperbolic, then by Theorem~\ref{T1}, $M$ is biholomorphic to $\Delta$. If $M$ is non-hyperbolic, then as observed in Subsection~\ref{S2}, $ M $ is biholomorphic either to $ \mathbb{C} $ or to $\mathbb{C}^{*}$. However, since $\Omega$ is simply connected so is $M$, and so $M$ must be biholomorphic to $\mathbb{C}$.

(ii) This is an immediate consequence of Corollary~\ref{T3}.

(iii) Since $\Omega$ is simply connected and non-hyperbolic, so is $M$, as observed in Subsection \ref{S2}. It follows that $M$ is biholomorphic to $\mathbb{C}$.

(iv) First, consider the case $n=1$. Without loss of generality, we may assume that $\Om=\mathbb{C}^{*}$. Then $\Omega$ is non-hyperbolic and hence $M$ is also non-hyperbolic, as observed earlier. Therefore, $M$ is biholomorphic either to $\mathbb{C}$ or to $\mathbb{C}^{*}$. By identifying $M$ with its image under this biholomorphism, we may assume that $M \subset\mathbb{C}$. Then $M_1\subset \mathbb{C}$ which is biholomorphic to $\mathbb{C}^{*}$, and hence $M_1=\mathbb{C} \setminus \{p\}$ for some $p \in \mathbb{C}$. Now, for each 
$j$, $M_1=\mathbb{C} \setminus \{p\} \subset M_j \subset \mathbb{C}$, and $M_j$ is biholomorphic to $\mathbb{C}^{*}$, and hence $M_j=\mathbb{C} \setminus \{p\}$. This implies that $M=\mathbb{C} \setminus \{p\}$, and thus $M$ is biholomorphic to $\mathbb{C}^{*}$.

Now consider the case $n \geq 2$. Then $\Omega$ is hyperbolic and hence taut as observed in Subsection~\ref{S2-0}.

\begin{claim}
$M$ is hyperbolic.
\end{claim}

\begin{proof}[Proof of the Claim]
If possible, assume that $M$ is non-hyperbolic. Then by remarks made in Subsection~\ref{S2}, $M$ is biholomorphic either to $\mathbb{C}$ or to $\mathbb{C}^{*}$. Identifying $M$ with its image under this biholomorphism, we may assume that $M \subset \mathbb{C}$. Then $M_1 \subset \mathbb{C}$ which is biholomorphic to $\mathbb{C} \setminus \{a_1, \ldots, a_n\}$, and hence $M_1=\mathbb{C} \setminus \{b_1, \ldots, b_n\}$, for $n$ distinct points $b_1, \ldots, b_n \in \mathbb{C}$. To see this, first note that $\pa_{\infty}\Omega$ has $n+1$ components and hence $\pa_{\infty}M_1$ also has $n+1$ components. Moreover, these components of $\pa_{\infty}M_1$ are precisely the cluster sets of $\phi_1: \Omega \to M_1$ at the points $a_k$, $0 \leq k \leq n$, where $a_0=\infty$ (see Proposition~1.8 and Proposition~1.11 in \cite[Chapter~14]{Con2}, and also note that the union of these cluster sets is $\pa_{\infty}M_1$). For each $k=0, \ldots,n$, let $\be_k$ be the cluster set of $\phi_1$ at $a_k$. Since $\phi_1 : \Omega \to M_1$ is injective, the singularity $a_k$ is either removable or a pole. Therefore, $\lim_{z \to a_k} \phi_1(z)$ exists in $\mathbb{C}_{\infty}$, and hence the  cluster set $\be_k$ must be a singleton. It follows that $M_1$ is of the above form. Now, for each $j\geq 1$, $M_1=\mathbb{C} \setminus \{b_1,\ldots, b_n\} \subset M_j \subset \mathbb{C}$, and $M_j$ is biholomorphic to $\mathbb{C}\setminus \{a_1, \ldots, a_n\}$, and hence $M_j=\mathbb{C} \setminus \{b_1, \ldots, b_n\}$. This implies that $M=\mathbb{C} \setminus \{b_1, \ldots, b_n\}$, which is a contradiction as the latter domain is hyperbolic because $n \geq 2$. This proves the claim. \phantom\qedhere\hfill $ \blacktriangleleft $
\end{proof}

Now that we have established that $M$ is hyperbolic, note that as in the proof of Theorem~\ref{T1}, the `orbit' $\{\psi_j(z_0)\}$ cannot accumulate at any `puncture' $a_k $. Here, as before, $ \psi_j: M_j \rightarrow \Omega $
are biholomorphisms and $ z_0 \in M $ is fixed. Next, we show that $\{\psi_j(z_0)\}$ cannot accumulate at $\infty$ either. Indeed, otherwise, identifying $\Omega$ with its image under the fractional linear transformation $z \mapsto 1/(z-a_1)$, we land in a situation where the orbit accumulates at a finite isolated boundary point of $\Omega$ which is impossible as noted above. Thus, the only possibility is that $ \{ \psi_j(z_0) \} $ is a relatively compact subset of $ \Omega $. In this setting, since $ \Omega $ is taut and $ M $ is hyperbolic, it follows from Lemma 3.1 of \cite{MV} that $M$ is biholomorphic to $\Omega$.
\end{proof}

\begin{rem}
It is only natural to consider the case $\Omega=\mathbb{C} \setminus A$, where $A$ is an infinite discrete subset of $\mathbb{C}$. The difficulty in solving the problem is that $\infty$ is a non-isolated boundary point of $\Omega$ in $\mathbb{C}_{\infty}$. Some possible scenarios are as follows: 

If $M$ is non-hyperbolic, then $M$ is biholomorphic either to $\mathbb{C}$ or to $\mathbb{C}^{*}$. Each of these cases is possible. To see the first case, let $M_j=\mathbb{C} \setminus \{n \in \mathbb{N} :n \geq j\}$. Then $M_j$ is an increasing sequence of domains in $\mathbb{C}$. Also, $\psi_j(z)=z-j+1$ are biholomorphisms of $M_j$ onto $\Omega=M_1$. Finally, note that $\cup M_j=\mathbb{C}$. For the second case, let $M_j=\mathbb{C}^{*} \setminus \{n \in \mathbb{N}:n \geq j\}$. Then $M_j$ is an increasing sequence of domains in $\mathbb{C}$. Also, $ \psi_j(z)=z-j+1$ are biholomorphisms of $M_j$ onto $\Omega=M_1$. Finally, note that $\cup M_j=\mathbb{C}^{*}$.

Let us now consider the case when $M$ is hyperbolic. Then, if $ \{ \psi_j(z_0) \} $ is a relatively compact subset of $ \Omega $ then as before, since $ \Omega $ is taut and $ M $ is hyperbolic, \cite[Lemma~3.1]{MV} ensures that $M$ is biholomorphic to $\Omega$. Also, note that $\{\psi_j(z_0)\}$ cannot accumulate at any point of $ A$ as in the proof of Theorem~\ref{T1}. Thus, we are left with the case when $\{\psi_j(z_0)\}$ accumulates at $\infty$. We do not know what happens in this case. It would be interesting and useful to clarify this point.
\end{rem}

\medskip

\noindent \textbf{Acknowledgements:} This project is motivated by a question raised by Gadadhar Misra during a conference talk and the authors are grateful to him. Sincere thanks are due to Kaushal Verma for several fruitful discussions throughout the course of this work. Thanks are also due to G. P. Balakumar for his insights that were instrumental in shaping this work. D. Borah is supported in part by an SERB grant (Grant No. CRG/2021/005884).

\end{document}